\documentclass[a4paper]{article}
\usepackage{amsmath,amssymb,amsthm,amsfonts}

\def \le {\leqslant}
\def \ge {\geqslant}
\def \PP {\mathbf P}
\def \NN {\mathbb N}
\def \RR {\mathbb R}
\def \ZZ {\mathbb Z}
\def \eps {\varepsilon}
\def \kappa {\varkappa}
\def \mf {\mathfrak}
\def \mc {\mathcal}
\def \ol {\overline}

\theoremstyle{plain}
\newtheorem{theorem}{Theorem}
\newtheorem{lem}{Lemma}
\newtheorem{cor}{Corollary}
\newtheorem{cor1}{Corollary}
\newtheorem{prop}{Proposition}

\begin{document}

\title{On distribution of fractional parts of linear forms}
\author{I.~Rochev}
\date{}

\maketitle

\section{Introduction}

In 1924, Khintchine proved (published in 1926, see~\cite[Hilfssatz
III]{H}) that, given an increasing sequence of positive integers
$\{q_n\}_{n=1}^\infty$, satisfying
$$\frac{q_{n+t}}{q_n}\ge2\qquad(n=1,2,\ldots)$$
for some $t\in\NN$, there exists a real number $\alpha$ such that
for all $n\in\NN$,
$$\|q_n\alpha\|>\gamma,$$
where $\gamma>0$ depends only on $t$. Here $\|x\|$ denotes the
distance from a real number $x$ to the nearest integer,
$\|x\|=\min\limits_{n\in\ZZ}|x-n|$.

Khintchine does not compute $\gamma$ but from his proof it is
clear that one can take
$$\gamma=\frac{c}{\bigl(t\ln(t+1)\bigr)^2}$$
with some absolute constant $c>0$.

The further history of the problem can be found, for instance,
in~\cite{M},\cite{M1}. Here we just mention the work~\cite{PS},
where a special variant of the Lov\'asz local lemma (see
Lemma~\ref{l1} below) is used to prove that one can take
$$\gamma=\frac{c}{t\ln(t+1)},$$
where $c>0$ is some absolute constant.

Similar results can be proved about the distribution of fractional
parts of linear forms. Thus, in~\cite[Chapter~V, Lemma~2]{K} the
following statement is demonstrated.

\textit{Let $\vec u_r=(u_{r1},\ldots,u_{rn})$, $r\in\NN$, be a
sequence of integer vectors, $\vec u_r\ne\vec0$. Assume that their
(Euclidean) norms
$$\rho_r=\left(u_{r1}^2+\ldots+u_{rn}^2\right)^{1/2}$$
satisfy
$$\rho_{r+1}\ge k\rho_r\qquad(r=1,2,\ldots)$$
for some $k>2$. Then there exists a vector
$\vec\alpha=(\alpha_1,\ldots,\alpha_n)\in\RR^n$, such that for all
$r\in\NN$,
$$\|\vec u_r\cdot\vec\alpha\|=\|u_{r1}\alpha_1+\ldots+u_{rn}\alpha_n\|\ge\frac12\left(1-\frac1{k-1}\right).$$}

In the present paper we use arguments from~\cite{PS}, as well as
from~\cite{M}, to obtain generalizations of the above-mentioned
result of Peres--Schlag and some results of the
work~\cite{M},\cite{M1} in the case of linear forms.
Section~\ref{sec1} contains some auxiliary results. In
Section~\ref{sec2} we introduce some notation and prove some
technical assertions, expounding the ideas of methods of
Peres--Schlag and Moshchevitin. Finally, in Section~\ref{sec3} we
apply these results to certain examples.

\section{Auxiliary assertions}\label{sec1}

\begin{lem}\label{l1}Let $\{A_n\}_{n=1}^N$ be events in a
probabilistic space $(\Omega,\mathcal F,\PP)$, and let
$\{x_n\}_{n=1}^N$ be a collection of numbers from $[0;1]$. Denote
$B_0=\Omega$, $B_n=\bigcap\limits_{m=1}^nA_m^c$ ($1\le n\le N$),
where $A_m^c=\Omega\setminus A_m$. Suppose that for every
$n\in\{1,\ldots,N\}$ there exists $m=m(n)\in\{0,1,\ldots,n-1\}$
such that
\begin{equation}\label{1}
\PP(A_n\cap B_m)\le x_n\prod_{m<k<n}(1-x_k)\cdot\PP(B_m)
\end{equation}
(if $m=n-1$, then $\prod\limits_{m<k<n}(1-x_k)=1$). Then for every
$1\le n\le N$,
\begin{equation}\label{2}
\PP(B_n)\ge(1-x_n)\PP(B_{n-1}).
\end{equation}
\end{lem}

\begin{proof}

We use induction on $n$.

\underline{Base} \underline{of} \underline{induction}.  One has
$$\PP(B_1)=1-\PP(A_1)\ge1-x_1=(1-x_1)\PP(B_0).$$

\underline{\mathstrut Inductive} \underline{\mathstrut step}.
Assume that (\ref{2}) is verified for $1\le n<n_0$. Using it
inductively for $n=n_0-1,n_0-2,\ldots,m+1$ (where $m=m(n_0)$), one
gets
$$\prod_{m<k<n_0}(1-x_k)\cdot\PP(B_m)\le\PP(B_{n_0-1}).$$
In view of (\ref{1}), one has
$$\PP(A_{n_0}\cap B_{n_0-1})\le\PP(A_{n_0}\cap B_m)\le x_{n_0}\PP(B_{n_0-1}),$$
hence
$$\PP(B_{n_0})=\PP(B_{n_0-1})-\PP(A_{n_0}\cap B_{n_0-1})\ge(1-x_{n_0})\PP(B_{n_0-1}).$$
Thus, (\ref{2}) holds for $n=n_0$.\end{proof}

Let $d\in\NN$, $\vec a=(a_1,\ldots,a_d)\in\RR^d$, $b\in\RR$,
$\eps>0$. Consider
$$E=E(d,\vec a,b,\eps)=\{\vec\theta\in[0;1]^d:\|\vec{ a}\cdot\vec\theta+b\|\le\eps\},$$
$V=V(d,\vec a,b,\eps)=\mu E$, where $\mu$ is the $d$-dimensional
Lebesgue measure. For $p\in[1;\infty]$, set
$$R=|\vec a|_p=\begin{cases}\left(\sum\limits_{n=1}^d|a_n|^p\right)^{1/p},&p\in[1;\infty);\\
\max\limits_{1\le n\le d}|a_n|,&p=\infty. \end{cases}$$

\begin{lem}\label{l2}
If $R>0$ then $V\le2\eps\left(1+\dfrac{d^{1/p}}R\right)$, where
$d^{1/p}=1$ for $p=\infty$.
\end{lem}

\begin{proof} If $\eps>1/2$ then the statement is trivial. Assume
that $\eps\le1/2$. Consider the cases $d=1$ and $d>1$ separately.

\underline{$d=1$}. It is easy to see that for any segment
$I\subset\RR$ of length $1/R$,
$$\mu\{\theta\in I : \|a\theta+b\|\le\eps\}=2\eps/R.$$
Since the segment $[0;1]$ can be covered by $\lceil R\rceil$
segments of length $1/R$, then
$$V\le2\eps/R\cdot\lceil R\rceil<2\eps(1+1/R).$$

\underline{$d>1$}. Without loss of generality we may assume that
$|a_1|=\max\limits_{1\le n\le d}|a_n|$, hence, $|a_1|\ge
R/d^{1/p}$. Using Fubini's theorem we get
$$V=\int\limits_{[0;1]^d}\chi_E(\vec\theta)\,d\mu=\int\limits_{[0;1]^{d-1}}\int\limits_0^1\chi_E(\vec\theta)\,d\theta_1\,d\mu',$$
where $\chi_E$ is the characteristic function of $E$, $\mu'$ is
the $(d-1)$-dimensional Lebesgue measure on variables
$\theta_2,\ldots,\theta_d$. Using the considered case one gets
$$\int\limits_0^1\chi_E(\vec\theta)\,d\theta_1=V\biggl(1,a_1,\sum\limits_{n=2}^{d}a_n\theta_n+b,\eps\biggr)\le2\eps\left(1+\frac1{|a_1|}\right)\le2\eps\left(1+\dfrac{d^{1/p}}R\right),$$
and the statement follows immediately .
\end{proof}

\begin{cor1}\label{c1}
Let $I=[v_1;v_1+r]\times\ldots\times[v_d;v_d+r]\subset\RR^d$ be
any cube with side $r>0$. Then
$$\frac{\mu\{\vec\theta\in I: \|\vec{ a}\cdot\vec\theta+b\|\le\eps\}}{\mu (I)}\le2\eps\left(1+\dfrac{d^{1/p}}{Rr}\right).$$
\end{cor1}

\begin{proof}
The statement follows from Lemma~\ref{l2} if one uses the linear
change of coordinates $\vec\theta=\vec v+r\vec\vartheta$,
$\vec\vartheta\in[0;1]^d$.
\end{proof}

\section{General results}\label{sec2}

Given $d\in\NN$ and sequences $\vec a_n\in\RR^d$, $b_n\in\RR$
($n\in\NN$), denote
$$L_n(\vec\theta)=L_n(\theta_1,\ldots,\theta_d)=\vec
a_n\cdot\vec\theta+b_n.$$ Fix $p\in[1;\infty]$. Assume that
$R_n=|\vec a_n|_p$ satisfy $$0<R_1\le R_2\le\ldots$$ We keep this
notation for the rest of the paper.

Suppose we also have a non-increasing sequence of positive numbers
$\delta_1\ge\delta_2\ge\ldots>0$. Consider the sets
$$\mathfrak G_1=\{\vec\theta\in\RR^d : \forall n\in\NN\quad\|L_n(\vec\theta)\|\ge\delta_n\};$$
$$\mathfrak G_2=\biggl\{\vec\theta\in\RR^d : \liminf\limits_{n\to\infty}\frac{\|L_n(\vec\theta)\|}{\delta_n}\ge1\biggr\}.$$

\begin{prop}\label{pr1}
Let $\lambda\in\RR$, $x_n\in(0;1)$ ($n\in\NN$). Suppose that for
every $n\in\NN$ there is
    $m=m(n)\in\{0,1,\ldots,n-1\}$ such that the following conditions hold:
    \begin{enumerate}
        \item\label{cond1} If $m>0$, then $R_n/R_m\ge2^{2\lambda+1}d/\delta_m$;
        \item\label{cond2} $2(1+2^{-\lambda})^2\delta_n\le x_n\prod\limits_{m<k<n}(1-x_k)$.
    \end{enumerate}
Then the set $\mathfrak G_1$ is non-empty. Moreover, if
$\lim\limits_{n\to\infty}R_n=\infty$, then the set $\mathfrak G_2$
is everywhere dense.
\end{prop}

\begin{proof}

First assume that $R_1\ge2^{|\lambda|}d^{1/p}$. Let us prove that
$\mathfrak G_1\cap[0;1]^d\ne\varnothing$.

Introduce some notation. Let $q\in[1;\infty]$ be the H\"older's
conjugate of $p$ (i.~e., $1/p+1/q=1$). Put $l_0=0$ and for
$n\in\NN$ define
$$l_n=\left\lceil\log_2\frac{d^{1/q}R_n}{\delta_n}+\lambda\right\rceil.$$
Notice that the sequence $l_n$ is non-decreasing.

Further, for $n\in\NN_0$ and $\vec c=(c_1,\ldots,c_d)\in\mathcal
C_n=\{0,1,\ldots,2^{l_n}-1\}^d$ put
$$I_n(\vec c)=\left[\frac{c_1}{2^{l_n}};\frac{c_1+1}{2^{l_n}}\right\rangle\times\ldots\times\left[\frac{c_d}{2^{l_n}};\frac{c_d+1}{2^{l_n}}\right\rangle,$$
where the notation $$\left[\frac
c{2^l};\frac{c+1}{2^l}\right\rangle=\begin{cases}\left[\frac
c{2^l};\frac{c+1}{2^l}\right),&c<2^l-1;\\
\left[\frac c{2^l};\frac{c+1}{2^l}\right],&c=2^l-1,\end{cases}$$
is used. Notice that for every $n\in\NN_0$ the cubes $I_n(\vec c)$
($\vec c\in\mathcal C_n$) are pairwise disjoint, and for any
integers $n\ge m\ge0$ every cube of the form $I_m(\vec c)$ can be
represented as a union of cubes of the form $I_n(\vec d)$.

For $n\in\NN$ consider
\begin{equation}\label{3}E_n=\{\vec\theta\in[0;1]^d:\|L_n(\vec\theta)\|<\delta_n\};\end{equation}
$$A_n=\bigsqcup_{\vec c\in\mathfrak C_n}I_n(\vec c),$$
where $\mathfrak C_n$ is the set of those vectors $\vec
c\in\mathcal C_n$, for which $I_n(\vec c)\cap E_n\ne\varnothing$.
Then $E_n\subset A_n$.

Let $\vec\theta\in A_n$. Then there is $\vec c\in\mathfrak C_n$
such that $\vec\theta\in I_n(\vec c)$, and there is $\vec \xi\in
I_n(\vec c)\cap E_n$. Therefore,
$$\|L_n(\vec\theta)\|=\|L_n(\vec\xi)+\vec
a_n\cdot(\vec\theta-\vec\xi)\|\le\|L_n(\vec\xi)\|+|\vec
a_n\cdot(\vec\theta-\vec\xi)|<$$ $$<\delta_n+|\vec
a_n|_p\cdot|\vec\theta-\vec\xi|_q\le\delta_n+R_nd^{1/q}2^{-l_n}\le(1+2^{-\lambda})\delta_n.$$
Thus, all vectors $\vec\theta\in A_n$ satisfy
$\|L_n(\vec\theta)\|<(1+2^{-\lambda})\delta_n$.

Define $B_n$, as in Lemma~\ref{l1}, assuming $\Omega=[0;1]^d$.

Let $n\in\NN$, $m=m(n)$. We check that~(\ref{1}) holds (with
$\PP=\mu$). The set $B_m$ can be represented in the form
$B_m=\bigsqcup\limits_{\vec c\in\mf D_m}I_m(\vec c)$, where $\mf
D_m$ is a subset of $\mathcal C_m$ (possibly, empty). Then
$$A_n\cap B_m=\bigsqcup_{\vec
c\in\mf D_m}(A_n\cap I_m(\vec c)).$$

Since (for any $\vec c\in\mc C_m$)
$$A_n\cap I_m(\vec c)\subset\{\vec\theta\in I_m(\vec c): \|L_n(\vec\theta)\|\le(1+2^{-\lambda})\delta_n\},$$
it follows from Corollary~\ref{c1} of Lemma~\ref{l2} that
$$\frac{\mu(A_n\cap I_m(\vec c))}{\mu (I_m(\vec c))}\le2(1+2^{-\lambda})\delta_n\left(1+\dfrac{d^{1/p}}{R_n2^{-l_m}}\right).$$

If $m=0$ then $\frac{d^{1/p}}{R_n2^{-l_m}}\le
d^{1/p}/R_1\le2^{-\lambda}$, because we assume that
$R_1\ge2^{|\lambda|}d^{1/p}$.

If $m>0$ then
$$\frac{d^{1/p}}{R_n2^{-l_m}}<2^{\lambda+1}\frac{d^{1/p+1/q}R_m}{R_n\delta_m }=\frac{2^{\lambda+1}dR_m}{\delta_mR_n}\le2^{-\lambda}$$
in view of Condition~\ref{cond1} of the proposition.

In any case
$$\frac{\mu(A_n\cap I_m(\vec c))}{\mu (I_m(\vec c))}\le2(1+2^{-\lambda})^2\delta_n\le x_n\prod\limits_{m<k<n}(1-x_k),$$
consequently,
$$\mu(A_n\cap B_m)\le x_n\prod\limits_{m<k<n}(1-x_k)\cdot\sum_{\vec c\in\mf D_m}\mu(I_m(\vec c))= x_n\prod\limits_{m<k<n}(1-x_k)\cdot\mu(B_m).$$

Thus, the inequality~(\ref{1}) holds. Hence, for any $n\in\NN$ one
has $\mu(B_n)\ge\prod\limits_{m=1}^n(1-x_m)>0$; in particular,
$B_n\ne\varnothing$. Denote
\begin{equation}\label{4}F_n=\bigcap\limits_{m=1}^nE_m^c,\end{equation} where $E_n$ are given by~(\ref{3}). Then
for every $n\in\NN$ the relation $F_n\supset B_n$ holds, hence
$F_n\ne\varnothing$. Since all $E_n^c$ are compact, it follows
that $\mf G_1\cap[0;1]^d=\bigcap\limits_{n=1}^\infty
F_n\ne\varnothing$.

If $R_1<2^{|\lambda|} d^{1/p}$ then make the linear change of
variables $\vec\theta=\frac{2^{|\lambda|}
d^{1/p}}{R_1}\cdot\vec\vartheta$. Using the proved one gets $\mf
G_1\ne\varnothing$.

Now we prove the second statement of the proposition. Let
$I=[v_1;v_1+r]\times\ldots\times[v_d;v_d+r]\subset\RR^d$ be any
cube with side $r>0$. Make the linear change of variables
$\vec\theta=\vec v+r\vec\vartheta$, $\vec\vartheta\in[0;1]^d$.

Since $\lim\limits_{n\to\infty}R_n=\infty$, there is $n_0\in\NN$
such that $rR_{n_0}\ge2^{|\lambda|}d^{1/p}$. Consider $\widetilde
L_n(\vec \theta)=L_{n_0-1+n}(r\vec\theta+\vec v)$ instead of
$L_n(\vec\theta)$, $\widetilde\delta_n=\delta_{n_0-1+n}$ instead
of $\delta_n$, $\widetilde x_n=x_{n_0-1+n}$ in place of $x_n$,
$\widetilde m(n)=\max\{m(n_0-1+n)-n_0+1;0\}$ instead of $m(n)$.
One deduces from what was proved that
$$\{\vec\theta\in I : \forall n\ge
n_0\quad\|L_n(\vec\theta)\|\ge\delta_n\}\ne\varnothing.$$ The
second assertion of the proposition follows
immediately.\end{proof}

\begin{prop}\label{pr2}
Let $\lambda\in\RR$, $\eta_\nu\in(0;1)$ ($\nu\in\NN_0$). Let
$\{n_\nu\}_{\nu\in\NN}$ be an increasing sequence of positive
integers. Denote
$$\sigma_\nu=\begin{cases}2(1+2^{-\lambda})\sum_{0<n\le n_1}\delta_n,&\nu=0;\\2(1+2^{-\lambda})^2\sum_{n_\nu<n\le n_{\nu+1}}\delta_n,&\nu\in\NN.\end{cases}$$
Suppose that the following is true:
\begin{enumerate}
    \item For $\nu\in\NN$
    $$\frac{R_{n_{\nu+1}+1}}{R_{n_\nu}}\ge\frac{2^{2\lambda+1}d}{\delta_{n_\nu}}.$$
    \item $$\sigma_0<\eta_0.$$
    \item For $\nu\in\NN$
      $$\sigma_\nu\le\eta_\nu(1-\eta_{\nu-1}).$$
    \item\label{cond3} There are infinitely many  $\nu\in\NN$ such
    that
    $$\left(1-\eta_{\nu}-\frac{\sigma_{\nu+1}}{\eta_{\nu+1}}\right)2^{d\lfloor\log_2Q_\nu\rfloor}\ge1,$$
    where
    $$Q_\nu=\frac{R_{n_{\nu+1}}\delta_{n_\nu}}{R_{n_\nu}\delta_{n_{\nu+1}}}.$$
\end{enumerate}
Then the set $\mathfrak G_1$ is of cardinality continuum. In
addition, the set $\mathfrak G_2$ is everywhere dense (moreover,
for any non-empty open set $\Omega\subset\RR^d$ the intersection
$\mf G_2\cap \Omega$ is of cardinality continuum).
\end{prop}

\begin{proof}
Take $R_0\ge2^{|\lambda|}d^{1/p}$ such that
$$(1+d^{1/p}/R_0)\sigma_0<\eta_0.$$
Note that for $\nu\in\NN$
$$(1+d^{1/p}/R_0)\frac{\sigma_\nu}{1+2^{-\lambda}}\le\sigma_\nu<\eta_\nu.$$
Let's prove that if $R_1\ge R_0$ then the set $\mathfrak
G_1\cap[0;1]^d$ is of cardinality continuum.

We preserve all the notation from the proof of
Proposition~\ref{pr1}. In addition, set $n_0=0$.

For $\nu\in\NN_0$ we define a \emph{$\nu$-cube} as a cube of the
form $I_{n_\nu}(\vec c)$, $\vec c\in\mc C_{n_\nu}$. We shall call
a $\nu$-cube $I$ \emph{good} if
$$\mu(B_{n_{\nu+1}}\cap I)>(1-\eta_\nu)\mu(I).$$

Let $\nu\in\NN$, $n_{\nu+1}<n\le n_{\nu+2}$. Then
$$\frac{R_n}{R_{n_\nu}}\ge\frac{R_{n_{\nu+1}+1}}{R_{n_\nu}}\ge\frac{2^{2\lambda+1}d}{\delta_{n_\nu}},$$
and the arguments, similar to those used in the proof of
Proposition~\ref{pr1}, give us that for any  $\nu$-cube $I$,
$$\frac{\mu(A_n\cap I)}{\mu(I)}\le2(1+2^{-\lambda})^2\delta_n.$$
Moreover, for $n\le n_2$,
$$\mu(A_n)\le2(1+2^{-\lambda})\delta_n(1+d^{1/p}/R_0)\le2(1+2^{-\lambda})^2\delta_n.$$
Therefore,
$$\mu(B_{n_1})\ge1-\sum_{n=1}^{n_1}\mu(A_n)\ge1-(1+d^{1/p}/R_0)\sigma_0>1-\eta_0,$$
i.~e., $[0;1]^d$ is a good $0$-cube.

Suppose that $\nu\in\NN$ and $I$ is a good $(\nu-1)$-cube. For
$n_{\nu}<n\le n_{n_{\nu+1}}$,
$$\mu(A_n\cap I)\le2(1+2^{-\lambda})^2\delta_n\mu(I)<\frac{2(1+2^{-\lambda})^2\delta_n}{1-\eta_{\nu-1}}\mu(B_{n_\nu}\cap I),$$
hence
$$\mu(B_{n_{\nu+1}}\cap I)\ge\mu(B_{n_\nu}\cap I)-\sum_{n_\nu<n\le n_{\nu+1}}\mu(A_n\cap I)>\left(1-\frac{\sigma_\nu}{1-\eta_{\nu-1}}\right)\mu(B_{n_\nu}\cap I).$$

Write $B_{n_\nu}\cap I$ in the form
$$B_{n_\nu}\cap I=\bigsqcup_{n=1}^aJ_n,$$
where $J_n$ are $\nu$-cubes. Then
$$a=\frac{\mu(B_{n_\nu}\cap I)}{2^{-dl_{n_\nu}}}>(1-\eta_{\nu-1})2^{d(l_{n_\nu}-l_{n_{\nu-1}})}.$$
Let $g$ denote the number of good $J_n$. Then
$$\left(1-\frac{\sigma_\nu}{1-\eta_{\nu-1}}\right)a2^{-dl_{n_\nu}}=\left(1-\frac{\sigma_\nu}{1-\eta_{\nu-1}}\right)\mu(B_{n_\nu}\cap I)<\mu(B_{n_{\nu+1}}\cap I)=$$
$$=\sum_{n=1}^a\mu(B_{n_{\nu+1}}\cap J_n)\le g2^{-dl_{n_\nu}}+(a-g)(1-\eta_{\nu})2^{-dl_{n_\nu}},$$
consequently,
$$g>\left(1-\frac{\sigma_{\nu}}{\eta_{\nu}(1-\eta_{\nu-1})}\right)a,$$
in particular, $g>0$. Hence, for every $\nu\in\NN_0$ any good
$\nu$-cube contains a good $(\nu+1)$-cube.

Further, if $\nu>1$, then
$$l_{n_\nu}-l_{n_{\nu-1}}>\log_2\frac{d^{1/q}R_{n_\nu}}{\delta_{n_\nu}}+\lambda-\left(\log_2\frac{d^{1/q}R_{n_{\nu-1}}}{\delta_{n_{\nu-1}}}+\lambda+1\right)=\log_2Q_{\nu-1}-1,$$
therefore,
$$g>\left(1-\eta_{\nu-1}-\frac{\sigma_{\nu}}{\eta_{\nu}}\right)2^{d(l_{n_\nu}-l_{n_{\nu-1}})}\ge\left(1-\eta_{\nu-1}-\frac{\sigma_{\nu}}{\eta_{\nu}}\right)2^{d\lfloor\log_2Q_{\nu-1}\rfloor}.$$
It follows now from Condition~\ref{cond3} of the proposition that
there are infinitely many  $\nu\in\NN$ such that every good
$\nu$-cube contains at least two good $(\nu+1)$-cubes. Thus, if we
denote by $G_\nu$ the union of closures of all good $\nu$-cubes,
then the set $G=\bigcap\limits_{\nu=0}^\infty G_\nu$ is of
cardinality continuum. Notice that
$$G_\nu\subset\ol{B_{n_\nu}}\subset\ol{F_{n_\nu}}=F_{n_\nu}$$
($\overline A$ denotes the closure of a set $A$, $F_n$ are given
by~(\ref{4})), therefore
$$G\subset\bigcap_{n=1}^\infty F_n=\mf G_1\cap[0;1]^d,$$
hence in the case $R_1\ge R_0$ the first statement of the
proposition is proved.

The rest of the proof is analogous to the end of the proof of
Proposition~\ref{pr1}. For that one should notice, that for
$\nu\in\NN$
$$\delta_{n_{\nu+1}}\le\frac{\sigma_\nu}{2(1+2^{-\lambda})^2}<\frac{1}{2(1+2^{-\lambda})^2},$$
hence
$$\frac{R_{n_{\nu+2}+1}}{R_{n_{\nu+1}}}\ge\frac{2^{2\lambda+1}d}{\delta_{n_{\nu+1}}}>4(2^\lambda+1)^2d>4,$$
thus $\lim\limits_{n\to\infty}R_n=\infty$.
\end{proof}

\section{Examples}\label{sec3}

\begin{theorem}
Suppose that there is $N\in\NN$ such that for any $n\in\NN$\quad
$R_{n+N}/R_n\ge2$. Denote
$$\delta=\frac1{2eN\Bigl(\log_2(Nd)+4\log_2\bigl(\log_2(Nd)+30\bigr)\Bigr)}.$$
Then the set
$$\{\vec\theta\in\RR^d : \inf_{n\in\NN}\|L_n(\vec\theta)\|\ge\delta\}$$
is non-empty. Moreover, the set
$$\{\vec\theta\in\RR^d : \liminf\limits_{n\to\infty}\|L_n(\vec\theta)\|\ge\delta\}$$
is everywhere dense.
\end{theorem}

\begin{proof}
Denote
$$u=\log_2(Nd)+30;$$
$$t=\log_2(Nd)+4\log_2u;$$
$$\lambda=\log_2(t\ln2);$$
$$h=\lceil\log_2(2^{2\lambda+1}d/\delta)\rceil;$$
$$x=\frac1{Nh}.$$ Apply Proposition~\ref{pr1}. Take $x_n=x$, $\delta_n=\delta$, $m(n)=\max\{0;n-Nh\}$. Then Condition~\ref{cond1} of Proposition~\ref{pr1} holds.
Since
$$\prod_{m<k<n}(1-x_k)\ge(1-1/(Nh))^{Nh-1}>\frac1e,$$
it is enough to verify that
$$2(1+2^{-\lambda})^2\cdot\delta\le x/e,$$
i.~e., $$\left(1+\frac1{t\ln2}\right)^2h\le t.$$ It is sufficient
to prove that $h\le t-2.9$. One has
$$h<\log_2\frac{2^{2\lambda+2}d}\delta=t-4\log_2u+3\log_2t+\log_2(8e\ln^22)<$$$$<t-4\log_2u+3\log_2(u-30+4\log_2u)+3.4<t-2.9.$$

Now the theorem follows from Proposition~\ref{pr1}.
\end{proof}

\begin{theorem}
Suppose that there is such $N\in\NN$ that for any $n\in\NN$\quad
$R_{n+N}/R_n\ge2$. Denote
$$\delta=\frac1{8N\Bigl(\log_2(Nd)+4\log_2\bigl(\log_2(Nd)+36\bigr)\Bigr)}.$$Then the set
$$\{\vec\theta\in\RR^d : \inf_{n\in\NN}\|L_n(\vec\theta)\|\ge\delta\}$$
is of cardinality continuum.
\end{theorem}

\begin{proof}
Denote
$$u=\log_2(Nd)+36;$$
$$t=\log_2(Nd)+4\log_2u;$$
$$\lambda=\log_2(t\ln2);$$
$$h=\lceil\log_2(2^{2\lambda+1}d/\delta)\rceil;$$
$$\eta=\frac{1+2^{-\lambda}}2\sqrt{\frac ht}.$$
One has
$$h<\log_2\frac{2^{2\lambda+2}d}\delta=t-4\log_2u+3\log_2t+\log_2(32\ln^22)<$$$$<t-4\log_2u+3\log_2(u-36+4\log_2u)+3.95<t-2.94;$$
$$2\eta<\left(1+\frac1{t\ln2}\right)\sqrt{1-\frac{2.94}t}<(1+1.45/t)(1-1.47/t)<1-0.02/t.$$

Apply Proposition~\ref{pr2}. Take $n_\nu=Nh\nu$,
$\delta_n=\delta$, $\eta_\nu=\eta$. Then
$$\sigma_0=\frac{\eta^2}{1+2^{-\lambda}};$$
$$\sigma_\nu=\eta^2.$$
It is clear that Conditions~1-3 of Proposition~\ref{pr2} hold.
Since for $\nu\in\NN$\quad $Q_\nu\ge2^h$, then
$$2^{d\lfloor\log_2 Q_\nu\rfloor}\ge2^h\ge\frac{2^{2\lambda+1}d}{\delta}=16\ln^22\cdot Ndt^3>100t.$$
Thus it is not difficult to see that Condition~\ref{cond3} is also
valid.

Proposition~\ref{pr2} now implies the theorem.
\end{proof}

\begin{theorem}\label{th3}
Let $f,h\colon[1;\infty)\to(0;\infty)$ be non-decreasing
functions, $h(x)\ge x$. Assume that
\begin{equation}\label{5}\lim_{x\to\infty}f(x)=\infty;\end{equation}
$$\sup_{x\ge1}\int\limits_x^{h(x)}\frac{du}{f(u)}<\infty;$$
\begin{equation}\label{6}\liminf_{n\to\infty}\frac{R_{\lfloor h(n)\rfloor}}{nf(n)R_n}>0.\end{equation}
Then the set
$$\{\vec\theta\in\RR^d : \inf_{n\in\NN}\bigl(\|L_n(\vec\theta)\|\cdot f(n)\bigr)>0\}$$
is of cardinality continuum. In addition, the set
$$\{\vec\theta\in\RR^d : \liminf\limits_{n\to\infty}(\|L_n(\vec\theta)\|\cdot f(n))>0\}$$
is everywhere dense.
\end{theorem}

\begin{proof}
Apply Proposition~\ref{pr2}. Take $\lambda=0$, $\eta_\nu=1/2$.
Take $n_1\in\NN$ large enough and define $n_{\nu+1}=\lfloor
h(n_\nu)\rfloor$, $\nu\in\NN$. Denote
$$C=\sup_{x\ge1}\int\limits_x^{h(x)}\frac{du}{f(u)};$$
$$A=A(n_1)=\max\{40Cf(n_1)/n_1;9\}.$$
Note that~(\ref{5}) implies
\begin{equation}\label{7}
A(n_1)=o(f(n_1))\text{\quad as }n_1\to\infty.
\end{equation}

Put
$$\delta_n=\begin{cases}\frac1{An_1},&n\le n_1;\\\frac{f(n_1)}{An_1f(n)},&n>n_1.\end{cases}$$
Then
$$\sigma_0=\frac4A<\frac12;$$
$$\sigma_\nu=\frac{8f(n_1)}{An_1}\sum_{n_\nu<n\le n_{\nu+1}}\frac1{f(n)}\le\frac{8f(n_1)}{An_1}\int\limits_{n_\nu}^{n_{\nu+1}}\frac{du}{f(u)}\le\frac{8Cf(n_1)}{An_1}\le\frac15\quad(\nu\in\NN).$$

By~(\ref{6}) there is a constant $\gamma>0$ such that for all
sufficiently large $n$,
$$\frac{R_{\lfloor h(n)\rfloor}}{R_n}\ge\gamma nf(n).$$
Hence, if $n_1$ is sufficiently large then, in view of~(\ref{7}),
one deduces that  for $\nu\in\NN$
$$\frac{R_{n_{\nu+1}+1}}{R_{n_\nu}}\ge\gamma n_\nu f(n_\nu)\ge\frac{2Ad}{f(n_1)}n_1f(n_\nu)=\frac{2d}{\delta_{n_\nu}}.$$
As long as
$$Q_\nu\ge\frac{R_{n_{\nu+1}}}{R_{n_{\nu}}}\to\infty,\quad\nu\to\infty,$$
all conditions of Proposition~\ref{pr2} hold.
\end{proof}

\begin{cor}
Suppose that
$$\liminf_{n\to\infty}\left(\frac{R_{n+1}}{R_n}-1\right)n^\beta>0,$$
where $\beta\in(0;1)$. Then the set
$$\{\vec\theta\in\RR^d : \inf_{n\in\NN}\left(\|L_n(\vec\theta)\|\cdot n^\beta\ln( n+1)\right)>0\}$$
is of cardinality continuum. In addition, the set
$$\{\vec\theta\in\RR^d : \liminf\limits_{n\to\infty}\left(\|L_n(\vec\theta)\|\cdot n^\beta\ln n\right)>0\}$$
is everywhere dense.
\end{cor}

\begin{proof}
Let
$$\gamma=\min\left\{1;\liminf_{n\to\infty}\left(\frac{R_{n+1}}{R_n}-1\right)n^\beta\right\}.$$
Take $f(x)=x^\beta\ln(x+1)$ and $h(x)=x+cx^\beta\ln (x+1)$,
$c=2/\gamma$. Then
$$\int\limits_x^{h(x)}\frac{du}{f(u)}\le\frac{h(x)-x}{f(x)}=O(1);$$
$$\ln\frac{R_{\lfloor h(n)\rfloor}}{R_n}\ge\sum_{k=n}^{\lfloor h(n)\rfloor-1}\ln\left(1+\frac{\gamma+o(1)}{k^\beta}\right)=\frac{\gamma+o(1)}{n^\beta}\cdot(h(n)-n+O(1))=(2+o(1))\ln n,\quad n\to\infty,$$
hence $$\lim_{n\to\infty}\frac{R_{\lceil
h(n)\rceil}}{nf(n)R_n}=\infty.$$ It remains to apply
Theorem~\ref{th3} .
\end{proof}

\begin{cor}
Assume that
$$\liminf_{n\to\infty}\left(\frac{R_{n+1}}{R_n}-1\right)n>0.$$
Then the set
$$\{\vec\theta\in\RR^d : \inf_{n\in\NN}\bigl(\|L_n(\vec\theta)\|\cdot n\ln( n+1)\bigr)>0\}$$
is of cardinality continuum. In addition, the set
$$\{\vec\theta\in\RR^d : \varliminf\limits_{n\to\infty}\bigl(\|L_n(\vec\theta)\|\cdot n\ln n\bigr)>0\}$$
is everywhere dense.
\end{cor}

\begin{proof}
The proof is similar. Take $f(x)=x\ln(x+1)$ and $h(x)=x^C$,
$C=3/\gamma+1$, where
$$\gamma=\min\left\{1;\liminf_{n\to\infty}\left(\frac{R_{n+1}}{R_n}-1\right)n\right\}.$$
Then
$$\int\limits_x^{h(x)}\frac{du}{f(u)}=O(1);$$
$$\ln\frac{R_{\lfloor h(n)\rfloor}}{R_n}\ge\sum_{k=n}^{\lfloor
h(n)\rfloor-1}\frac{\gamma+o(1)}{k}=(1+o(1))\gamma(C-1)\ln
n=(3+o(1))\ln n,\quad  n\to\infty.$$
\end{proof}

\begin{cor}
Assume that
$$\ln R_n=\gamma n^\beta+O(n^{\beta_1})\text{\quad as\ \ }n\to\infty,$$
where $\gamma>0$, $0\le\beta_1<\beta\le1$ are some constants.
Define
$$\alpha(x)=\begin{cases}1,&\beta_1>0;\\
\ln(x+1),&\beta_1=0.\end{cases}$$ Then the set
$$\{\vec\theta\in\RR^d : \inf_{n\in\NN}\left(\|L_n(\vec\theta)\|\cdot n^{1-\beta+\beta_1}\alpha(n)\right)>0\}$$
is of cardinality continuum. Moreover, the set
$$\{\vec\theta\in\RR^d : \varliminf\limits_{n\to\infty}\left(\|L_n(\vec\theta)\|\cdot n^{1-\beta+\beta_1}\alpha(n)\right)>0\}$$
is everywhere dense.
\end{cor}

\begin{proof} Let for $n\in\NN$
$$|\ln R_n-\gamma n^\beta|\le An^{\beta_1}.$$
Take $f(x)=x^{1-\beta+\beta_1}\alpha(x)$ and $h(x)=x+(C+1)f(x)$,
$C=\frac2{\beta\gamma}(3A+2)$. Then for all sufficiently large
$n$,
$$\ln\frac{R_{\lfloor h(n)\rfloor}}{R_n}>\gamma n^\beta\left((1+Cf(n)/n)^\beta-1\right)-3An^{\beta_1}>\left(\frac{\beta\gamma C}2\alpha(n)-3A\right)n^{\beta_1}\ge2n^{\beta_1}\alpha(n)>2\ln n.$$
\end{proof}


\begin{thebibliography}{10}

\bibitem{H}
A. Ya. Khintchine, \textit{\"Uber eine Klasse linearer
Diophantischer Approximationen}, Rendiconti Circ. Mat. Palermo
\textbf{50} (1926), 170-195.



\bibitem{PS}
Y. Peres and W. Schlag, \textit{Two Erd\"os problems on lacunary
sequences: Chromatic number and Diophantine approximation},
preprint, available at: arXiv:0706.0223v1 [math.CO] 1Jun 2007.

\bibitem{M}
Moshchevitin N.G.\,\,\,
 A version of the proof for  Peres-Schlag's theorem on lacunary sequences.
//
Preprint, available at arXiv: 0708.2087v2 [math.NT] 15Aug2007




\bibitem{M1}
Moshchevitin N.G.\,\,\, Density modulo 1 of sublacunary sequences:
application of
 Peres-Schlag's arguments.
  //
 Preprint, available at arXiv:  0709.3419v2 [math.NT] 20Oct2007

\bibitem{K}
J. W. S. Cassels, \textit{An Introduction to Diophantine
Approximation}, Cambridge Tracts no. 45, Cambridge University
Press, London, 1957.

\end{thebibliography}
\end{document}